\documentclass[
final
]{dmtcs-episciences}

\usepackage[utf8]{inputenc}
\usepackage{subfigure}

\usepackage{amsthm}
\usepackage{amsmath}

\usepackage[shortlabels]{enumitem}
\usepackage{algorithmic}
\usepackage{algorithm}
\algsetup{indent=2em}

\usepackage{tikz}
\usetikzlibrary{patterns,arrows,decorations.pathreplacing}

\theoremstyle{plain}
\newtheorem{theorem}{Theorem}[section]
\newtheorem{observation}[theorem]{Observation}
\newtheorem{proposition}[theorem]{Proposition}

\newtheorem{lemma}[theorem]{Lemma}
\newtheorem{corollary}[theorem]{Corollary}

\theoremstyle{definition}

\newtheorem{example}[theorem]{Example}

\newcommand{\Av}{\operatorname{Av}}
\newcommand{\C}{\mathcal{C}}
\newcommand{\st}{\::\:}
\newcommand{\type}[1]{T(#1)}

\newcommand{\la} 
	{
		\begin{tikzpicture}[scale=.2, anchor=base]
			\draw[fill, ultra thin] (0,0)--(0.5,0.25)--(0.5,-0.25)--cycle;
 		\end{tikzpicture}
	}
\newcommand{\ra} 
	{
		\begin{tikzpicture}[scale=.2, xscale=-1, anchor=base]
			\draw[fill, ultra thin] (0,0)--(0.5,0.25)--(0.5,-0.25)--cycle;
 		\end{tikzpicture}
	}
\newcommand{\ua} 
	{
		\begin{tikzpicture}[scale=.2, anchor=base]
			\draw[fill, ultra thin] (0,0)--(0.5,0)--(0.25,0.5)--cycle;
 		\end{tikzpicture}
	}
\newcommand{\da} 
	{
		\begin{tikzpicture}[scale=.2, yscale=-1, anchor=base]
			\draw[fill, ultra thin] (0,0)--(0.5,0)--(0.25,0.5)--cycle;
 		\end{tikzpicture}
	}

\newcommand{\cprob}[3]{
    \begin{center}
      \fbox{
        \parbox{0.95\textwidth}{
          #1\smallskip\\
          \begin{tabular}{rp{0.73\textwidth}}
            \textit{Input:\ } & #2\\
            \textit{Question:\ } & #3
          \end{tabular}
        }
      }
    \end{center}
}

\usepackage[square,sort,comma,numbers]{natbib}

\title{The Complexity of Pattern Matching for $321$-Avoiding and Skew-Merged Permutations}

\author
	[Michael Albert, Marie-Louise Lackner, Martin Lackner, and Vincent Vatter]
	{
	\begin{tabular}{ccc}
	Michael Albert\affiliationmark{1}
	&\quad\quad&
	Marie-Louise Lackner\affiliationmark{2}\thanks{M.-L. Lackner's research was supported by the Austrian Science Foundation FWF, grant P25337-N23.}
	\\
	Martin Lackner\affiliationmark{3}\thanks{M. Lackner's research was supported by the Austrian Science Foundation FWF, grants P25518-N23 and Y698, and by the European Research Council (ERC) under grant number 639945 (ACCORD).}
	&\quad\quad&
	Vincent Vatter\affiliationmark{4}\thanks{Vatter's research was partially supported by the National Science Foundation under Grant Number DMS-1301692.}
	\end{tabular}
	}

\affiliation{
	Department of Computer Science, University of Otago, Dunedin, New Zealand\\
	Institute of Discrete Mathematics and Geometry, TU Wien, Vienna, Austria\\
	Department of Computer Science, University of Oxford, Oxford, UK\\
	Department of Mathematics, University of Florida, Gainesville, Florida, USA
	}


\received{2015-10-22} 
\revised{2016-12-05} 
\accepted{2016-12-20}
\begin{document}
\publicationdetails{18}{2016}{2}{11}{1308}

\maketitle

\begin{abstract}
The \textsc{Permutation Pattern Matching} problem, asking whether a pattern permutation $\pi$ is contained in a text permutation $\tau$, is known to be \textsf{NP}-complete. We present two polynomial time algorithms for special cases. The first is applicable if both $\pi$ and $\tau$ are $321$-avoiding while the second is applicable if both permutations are skew-merged. Both algorithms have a runtime of $O(kn)$, where $k$ is the length of $\pi$ and $n$ the length of $\tau$.
\end{abstract}

\section{Introduction}

In this paper, a permutation is a bijective function from $[n]$ to itself, where $n$ is a positive integer and $[n] = \{1,2,\dots,n\}$. Therefore, a permutation, $\pi : [n] \to [n]$ \emph{is} the set of ordered pairs $(i, \pi(i))$. We occasionally write specific permutations in the usual one line notation, e.g., $321$ represents the permutation of $[3]$ equal to $\{(1,3), (2,2), (3,1)\}$. The \emph{size} of $\pi$ is just the cardinality of this set, and we denote the elements, also called points, of a permutation by variables such as $x$ and $y$. We adopt the usual conventions with respect to order of such points, i.e., $(i, \pi(i))$ lies to the left of $(j, \pi(j))$ if $i < j$ and above $(k, \pi(k))$ if $\pi(i) > \pi(k)$, with corresponding definitions for `to the right of' and `below'. Given an element $x$ in a permutation $\pi$, we define (wherever possible): \\
\begin{tabular}{ll}
$x^{\la}$ & the element immediately to its left, \\
$x^{\ra}$ & the element immediately to its right, \\
$x^{\ua}$ & the element immediately above it, and \\
$x^{\da}$ & the element immediately below it.
\end{tabular} \\
 We use the symbol $\bot$ to represent `undefined'. Any operator applied to $\bot$ also yields $\bot$. 
 For example, in the permutation $31254$ and for $x=3$ we have: $x^{\la}=\bot$, $x^{\ra}=1$, $x^{\ua}=4$ and $x^{\da}=2$.

Let $\pi$ and  $\tau$ be two permutations. An injective function, $f$, from $\pi$ into $\tau$ is an \emph{embedding} if, for all elements $x$ and $y$ of $\pi$, the elements $f(x)$ and $f(y)$ of $\tau$ are in the same relative order as $x$ and $y$ (e.g., if $x$ lies below and to the right of $y$, then $f(x)$ also lies below and to the right of $f(y)$). \label{def-embedding} If there is an embedding from $\pi$ into $\tau$ then we say that $\tau$ \emph{contains} $\pi$. If not, then we say that $\tau$ \emph{avoids} $\pi$. The following problem, analogous to problems related to detecting occurrences of patterns in words, is central with respect to these concepts:

\cprob
	{\textsc{Permutation Pattern Matching} (PPM)}
	{A text permutation $\tau$ of size $n$ and a pattern $\pi$ of size $k$.}
	{Does $\tau$ contain $\pi$?}

Bose, Buss, and Lubiw~\cite{bose:pattern-matchin:} showed in 1998 that the PPM problem is \textsf{NP}-complete\footnote{Note that because both the pattern and text are regarded as input, the size of the input is $n+k$. Were we to regard the size of the pattern as fixed, then the trivial $\mathcal{O}\left({n\choose k}\cdot k\right)$ algorithm would be polynomial time.}.

The recent work of Guillemot and Marx~\cite{guillemot:finding-small-p:} shows that the PPM problem can be solved in time $2^{O(k^2\log k)}n$, i.e., linear-time in $n$ when $k$ is viewed as a constant. In particular, this implies that the PPM problem is \emph{fixed-parameter tractable (fpt)} with respect to the size of the pattern $k$. Work prior to this breakthrough result achieved runtimes of $O(n^{1+2k/3}\cdot\log n)$~\cite{albert:algorithms-for-:} and $O(n^{0.47k+o(k)})$~\cite{ahal:on-complexity-o:}. Using $\mathsf{run}(\tau)$ as parameter, i.e., the number of alternating runs of $\tau$, Bruner and Lackner~\cite{bruner:a-fast-algorith:} present an fpt algorithm with runtime $O(1.79^{\mathsf{run}(\tau)}\cdot kn)$.

A \emph{permutation class}, $\C$, is a set of permutations with the property that if $\tau \in \C$ and $\tau$ contains $\pi$ then $\pi \in \C$. In other words, $\C$ is closed downwards with respect to the partial ordering of permutations given by the relation ``is contained in''. A permutation class is proper if it does not contain every permutation. Permutation classes are frequently defined in terms of avoidance conditions, namely, for any set $B$ of permutations, the set $\Av(B)$ consisting of those permutations which avoid every element of $B$ is a permutation class (and it is always proper if $B$ is non-empty). Conversely, every permutation class is equal to the class of permutations avoiding its complement, or even avoiding the minimal elements (with respect to the partial ordering mentioned above) of its complement. For an overview of results on permutation classes, we refer to the corresponding chapter in the \textit{Handbook of Enumerative Combinatorics}~\cite{vatter:permutation-cla:}.

This leads to two natural ways in which one might restrict the PPM problem. One is to impose additional structure on the pattern, most naturally, to insist that the pattern belongs to a particular (proper) permutation class. One example which has been studied is the class of \textit{separable} permutations, which are those avoiding both $3142$ and $2413$. If the pattern is separable, PPM can be solved in polynomial time~\cite{bose:pattern-matchin:,ibarra:finding-pattern:,albert:algorithms-for-:,yugandhar:parallel-algori:,neou:pattern-matchin:}. The fastest algorithm for this case is by Ibarra~\cite{ibarra:finding-pattern:} with a runtime of $O(kn^4)$. Formally, we define this class of problems as follows:

\cprob
	{$\C$-\textsc{Pattern Permutation Pattern Matching} (\textsc{$\C$-pattern PPM})}
	{A text permutation $\tau$ of size $n$ and a pattern $\pi$ of size $k$, where $\pi$ belongs to a fixed proper permutation class $\C$.}
	{Does $\tau$ contain $\pi$?}

A second and more restrictive specialisation of the PPM problem is to insist that both the pattern and text belong to a (proper) permutation class. This is the version of the problem that we study. 

\cprob
	{\textsc{$\C$ Permutation Pattern Matching} ($\C$-PPM)}
	{A text permutation $\tau$ of size $n$ and a pattern $\pi$ of size $k$, both belonging to a fixed proper permutation class $\C$.}
	{Does $\tau$ contain $\pi$?}
	
Clearly, for a fixed $\C$, polynomial time algorithms for \textsc{$\C$-pattern PPM} apply to  \textsc{$\C$-PPM} as well. Consequently, the separable case, i.e., \textsc{$\Av(3142, 2413)$-PPM}, can be solved in $O(kn^4)$ time~\cite{ibarra:finding-pattern:}.
Note that if the pattern avoids $132$, $231$, $213$ or $312$ then it is automatically separable and thus the \textsc{$\C$-pattern PPM} problem for all four classes $\Av(132)$, $\Av(231)$, $\Av(213)$ or $\Av(312)$ can be solved in polynomial time.
 Most relevant to our work is a result by Guillemot and Vialette~\cite{guillemot:pattern-matchin:} that establishes an $O(k^2n^6)$-time algorithm for $\Av(321)$-PPM. In Sections~\ref{sec-lattice} and \ref{sec-fluid}, we improve their approach to give the following.

\begin{theorem}
\label{thm-complexity-321}
Given $321$-avoiding permutations $\tau$ of size $n$ and $\pi$ of size $k$, there is an $O(kn)$-time algorithm which determines whether $\tau$ contains $\pi$.
\end{theorem}

In Section~\ref{sec-skew-merged} we show how to adapt this approach to the class of skew-merged permutations, which are those permutations whose elements can be partitioned into an increasing subsequence and a decreasing subsequence. Skew-merged permutations can also be characterised as those permutations that avoid both $3412$ and $2143$~\cite{stankova:forbidden-subse:}.

\begin{theorem}
\label{thm-complexity-skew-merged}
Given skew-merged permutations $\tau$ of size $n$ and $\pi$ of size $k$, there is an $O(kn)$-time algorithm which determines whether $\tau$ contains $\pi$.
\end{theorem}

The following elementary observation will be used repeatedly.

\begin{lemma}
\label{obs-check-neighbors}
Let $\pi$ and $\tau$ be permutations and $f: \pi\to \tau$. Then $f$ is an embedding of $\pi$ into $\tau$ if and only if for every element $x$ of $\pi$:\begin{itemize}[noitemsep, topsep=-10pt]
\item if $x^{\la} \neq \bot$ then $f(x)$ lies strictly to the right of $f(x^{\la})$ and
\item if $x^{\da} \neq \bot$ then $f(x)$ lies strictly above $f(x^{\da})$.
\end{itemize}
\end{lemma}
\begin{proof}
Suppose that $x$ and $y$ are points of $\pi$ and that, without loss of generality, $y$ lies strictly to the left of $x$. Then $y$ occurs in the sequence $x^{\la},x^{\la\la}, x^{\la\la\la} ,\dots$. So, by inductive use of the first property, $f(x)$ lies strictly to the right of $f(y)$. Similarly, inductive use of the second property establishes that the vertical relationship between $f(x)$ and $f(y)$ is the same as that between $x$ and $y$, and the result follows.

The other direction follows directly from the definition of embeddings given on page~\pageref{def-embedding}: $x^{\la}$ is an element strictly to the left of $x$ and thus $f(x)$ lies strictly to the right of $f(x^{\la})$ for an embedding $f$. In the same way, $x^{\da}$ is an element strictly below $x$ and thus $f(x)$ lies strictly above $f(x^{\da})$.
\end{proof}

\section{The Lattice of Rigid Embeddings of $321$-Avoiding Permutations}
\label{sec-lattice}

It is easy to see that the elements of any $321$-avoiding permutation $\pi$ can be partitioned into two increasing subsequences. This partition is in general not unique but in any such partition, one of these subsequences will contain all those elements which participate as the `$2$' in a copy of $21$---called the \emph{upper elements} of $\pi$ and denoted $U_\pi$--- and the other will contain all those elements which participate as the `$1$' in a copy of $21$---called the \emph{lower elements} of $\pi$ and denoted $L_\pi$. 
Elements that are neither upper nor lower elements, i.e., those that are not involved in a copy of $21$, can be part of either of the two subsequences.
Let us formalise these definitions: An element $x$ of $\pi$ is an upper element if there is some embedding of $21 = \{(1,2),(2,1)\}$ into $\pi$ such that $x$ is the image of $(1,2)$ and a lower element if there is an embedding of $21$ such that $x$ is the image of $(2,1)$.

Following Albert, Atkinson, Brignall, Ru\v{s}kuc, Smith, and West~\cite{albert:growth-rates-fo:}, elements which are either upper or lower elements of $\pi$ are referred to as \emph{rigid} elements, and $\pi$ is called a \emph{rigid permutation} if all of its elements are rigid (i.e., if $\pi = U_\pi \cup L_\pi$). The remaining elements will be called \emph{fluid} elements.
For an example of a $321$-avoiding permutation and its decomposition into rigid and fluid elements, see Figure~\ref{fig-321}.

Note that it can be determined in linear time which elements are upper, lower and fluid in a permutation. For this purpose one simply needs to scan the permutation from left to right and record the largest element encountered so far, denoted by $\ell$, and the smallest element yet to come at the right, denoted by $s$. When we read an element $x$, three cases can occur:
\begin{itemize}[noitemsep, nolistsep]
\item $x > s$: In this case $x \, s$ forms a $21$-pattern and thus $x$ is an upper element.
\item $x < \ell$: In this case $\ell \, x$ forms a $21$-pattern and thus $x$ is a lower element
\item $x \leq s$ and $x \geq \ell$ (which implies that $x=s$ and $s>\ell$): In this case $x$ does not occur in a $21$-pattern and is thus a fluid element.
\end{itemize}

\begin{figure}
\begin{center}
\usetikzlibrary{shapes, decorations.pathreplacing}
\begin{tikzpicture}[scale=0.3]
	\tikzstyle{u}=[rectangle, draw, fill=black, inner sep=0.0675cm]
	\tikzstyle{l}=[circle, draw, fill=black, inner sep=0.0525cm]
	\tikzstyle{f}=[star, fill=black, inner sep=0.045cm]
	\foreach \i in {1,...,13} {
		\draw [darkgray] (\i,0.5) -- (\i,13.5);
	}
	\foreach \i in {1,...,13} {
		\draw [darkgray] (0.5,\i) -- (13.5,\i);
	}

	\node[u] at (1,3) {};
	\node[l] at (2,1) {};
	\node[l] at (3,2) {};
	\node[f] at (4,4) {};
	\node[f] at (5,5) {};
	\node[u] at (6,9) {};
	\node[l] at (7,6) {};
	\node[l] at (8,7) {};
	\node[u] at (9,10) {};
	\node[l] at (10,8) {};
	\node[f] at (11,11) {};
	\node[u] at (12,13) {};
	\node[l] at (13,12) {};
	
	\draw[draw=black, very thick] (0.5,0.5) rectangle (3.5,3.5);
	\draw[draw=black, very thick] (3.5,3.5) rectangle (4.5,4.5);
	\draw[draw=black, very thick] (4.5,4.5) rectangle (5.5,5.5);
	\draw[draw=black, very thick] (5.5,5.5) rectangle (10.5,10.5);
	\draw[draw=black, very thick] (10.5,10.5) rectangle (11.5,11.5);
	\draw[draw=black, very thick] (11.5,11.5) rectangle (13.5,13.5);
	
	\draw (16,1) rectangle (25.75,5);
	\node[u] at (17,4)  {};
	\node[l] at (17,3) {};
	\node[f] at (17,2) {};
	\node at (18,4) [right] {upper};
	\node at (18,3) [right]  {lower};
	\node at (18,2) [right] {fluid};
	\draw [decorate,decoration={brace,amplitude=5pt,mirror,raise=4pt},yshift=0pt](21.25,2.5) -- (21.25,4.5);
	\node at (22.25,3.5) [right] {rigid};
\end{tikzpicture}
\end{center}
\caption{The decomposition of the $321$-avoiding permutation $\pi = 3 \, 1 \, 2 \, 4 \, 5 \, 9 \, 6 \, 7 \, 10 \, 8 \, 11 \, 13 \, 12$ into rigid and fluid elements.}
\label{fig-321}
\end{figure}
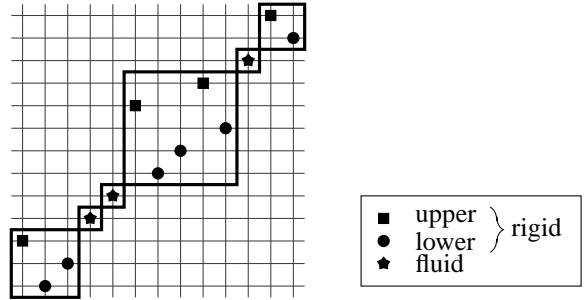

The existence of fluid elements in a pattern will be the source of some difficulty in solving the $\Av(321)$-PPM problem, and will be addressed in the next section. For the remainder of this section we consider a rigid pattern $\pi$ of size $k$ and a $321$-avoiding text $\tau$ of size $n$. Since an embedding preserves relative locations of points, the image of any rigid element must be rigid. More precisely, we have the following:

\begin{observation}
Let $\pi$ be a rigid pattern and $\tau$ be an arbitrary $321$-avoiding permutation.
If there exists an embedding of $\pi$ into $\tau$, then it must map upper (resp., lower) elements of $\pi$ to upper (resp., lower) elements of $\tau$ and the fluid elements of $\tau$ will never occur in an embedding.
\end{observation}

In order to look for such embeddings we must widen our search space. A map $f:\pi \to \tau$ is called a \emph{rigid mapping} if $f$ maps upper (resp., lower) elements of $\pi$ to upper (resp., lower) elements of $\tau$. As noted above, because $\pi$ is rigid, every embedding of $\pi$ into $\tau$ is a rigid mapping, but the converse is far from true since, among other reasons, rigid mappings need not be injective.

Given two points, $x$ and $y$, in $U_\pi$, we say $x \leq y$ if $y$ lies above and to the right of $x$. This is a linear order on $U_\pi$, and we have similar linear orders (all denoted $\leq$) on $L_\pi$, $U_\tau$ and $L_\tau$. This makes the set of all rigid mappings of $\pi$ into $\tau$ into a partially ordered set using point-wise comparison; that is, given rigid mappings $f,g \st \pi\rightarrow\tau$, we write $f\le g$ if $f(x) \leq g(x)$ for all elements $x$ of $\pi$. In fact, it is easy to see that this partially ordered set is a distributive lattice; given two rigid mappings $f,g : \pi \to \tau$ their \emph{meet} and \emph{join} can be defined, respectively, by
\begin{eqnarray*}
	(f\wedge g)(x)
	&=&
	\min\{f(x),g(x)\},\\
	(f\vee g)(x)
	&=&
	\max\{f(x),g(x)\}
\end{eqnarray*}
for all elements $x$ of $\pi$.

%

It is notable that these observations also hold for embeddings.
That is, the set of embeddings from $\pi$ into $\tau$ is a sublattice of the lattice of rigid mappings:
\begin{theorem}bruner
[{\cite[Theorem 2]{albert:growth-rates-fo:}}]
\label{thm-321-lattice}
Given a rigid pattern $\pi$ and a $321$-avoiding text $\tau$, the set of embeddings of $\pi$ into $\tau$ forms a distributive lattice under the operations of meet and join defined above.
\end{theorem}

It follows from Theorem~\ref{thm-321-lattice} that if $\pi$ is contained in $\tau$ then there is a \emph{minimum embedding} of $\pi$ into $\tau$ which we denote by $e_{\min}$.

Given an element $x$ of some $321$-avoiding permutation $\sigma$ we define $x^{\la}_U$ to be the rightmost element of $U_\sigma$ that is to the left of $x$. Of course, we have corresponding notations such as $x^{\da}_L$, $x^{\ra}_U$ and so on. In all cases, if no such element exists we get $\bot$ as usual. We also define the \emph{type} of $x$, $\type{x}$ to be $U$ if $x \in U_\sigma$ and $L$ if $x \in L_\sigma$. The following result forms the core of our algorithm for determining whether there is an embedding of $\pi$ into $\tau$, at least for the case where $\pi$ is rigid.
It will allow us to turn an arbitrary rigid mapping into an embedding, if possible.

\begin{proposition}
\label{prop-allowed-update}
Suppose that $e : \pi \to \tau$ is an embedding, $f : \pi \to \tau$ is a rigid mapping, and, for all $x \in \pi$, $f(x) \leq e(x)$. Then, for all $x \in \pi$:
\[
\max \{ f(x^{\la})^{\ra}_{\type{x}}, f(x^{\da})^{\ua}_{\type{x}} \} \leq e(x),
\]
where we define $\max\{y,\bot\}=y$ and $\max\{\bot\}=0$.
\end{proposition}

\begin{proof}
We first establish that $f(x^{\la})^{\ra}_{\type{x}} \leq e(x)$. Since $x$ lies strictly to the right of $x^{\la}$ (and $e$ is an embedding), $e(x)$ lies strictly to the right of  $e(x^{\la})$ and is of the same type as $x$, so, it does not lie to the left of $e(x^{\la})^{\ra}_{\type{x}}$. Consequently, $e(x^{\la})^{\ra}_{\type{x}} \leq e(x)$. But $f(x^{\la}) \leq e(x^{\la})$ and so $ f(x^{\la})^{\ra}_{\type{x}} \leq e(x^{\la})^{\ra}_{\type{x}} \leq e(x)$. The arguments for the other case are exactly the same.
\end{proof}

Applying the proposition above in the case where $f = e$, we see that for any embedding, $e$, from a rigid $\pi$ into $\tau$, and any $x \in \pi$:
\[
\max \{ e(x^{\la})^{\ra}_{\type{x}}, e(x^{\da})^{\ua}_{\type{x}} \} \leq e(x).
\]
Now suppose that $f$ is any rigid mapping from $\pi$ to $\tau$. We say that $x$ \emph{is a problem} if it violates the above condition, i.e., $x$ is a problem if 
\begin{equation}\label{eqn:problem-condition}
f(x) < \max \{ f(x^{\la})^{\ra}_{\type{x}}, f(x^{\da})^{\ua}_{\type{x}} \}.
\end{equation}
Intuitively , $x$ is a problem if $f(x)$ is too low compared with $f(x^{\da})$ or too far left compared with $f(x^{\la})$.
We let $P(f)$ be the set of problems for $f$, for which the following holds:

\begin{corollary}
Let $\pi$ be a rigid permutation and $\tau$ a $321$-avoiding permutation.
A rigid mapping $f$ is an embedding of $\pi$ into $\tau$ if and only if the set of problems $P(f)$ is empty.
\end{corollary}
\begin{proof}
If $f$ is an embedding, it follows from Proposition~\ref{prop-allowed-update} that no element $x\in\pi$ fulfills condition~\eqref{eqn:problem-condition}. Thus $P(f)$ is empty.

For the other direction, assume that $f$ is not an embedding. From Lemma~\ref{obs-check-neighbors} we know that there exists an $x\in\pi$ such that $f(x)$ is below or equal to $f(x^{\da})$ or such that $f(x)$ is left of or equal to $f(x^{\la})$. 
First, if $f(x)$ is below or equal to $f(x^{\da})$ this implies that $f(x)$ is strictly below $f(x^{\da})^{\ua}_{\type{x}}$ and hence $x\in P(f)$.
Second, if $f(x)$ is left of or equal to $f(x^{\la})$, we have that $f(x)$ is strictly left of $f(x^{\la})^{\ra}_{\type{x}}$.
Now note that for $f(x), f(y)\in \tau$ of the same type, $f(x)$ is left of $f(y)$ if and only if $f(x)$ is below $f(y)$.
Moreover we know that $f$ preserves types and thus $f(x)$ and $f(x^{\la})^{\ra}_{\type{x}}$ have the same type.
We conclude that $f(x)<f(x^{\la})^{\ra}_{\type{x}}$ and thus $x\in P(f)$.
\end{proof}

We now describe an algorithm, displayed as Algorithm \ref{alg-rigid-case}. Given as input a rigid permutation $\pi$ and a $321$-avoiding permutation $\tau$, it returns the minimum embedding $e_{\min}$  of $\pi$ into $\tau$ when it exists, and fails otherwise. The algorithm constructs and updates a rigid mapping $f$, ensuring that $f \leq e_{\min}$ at all times (if an embedding exists).
Let $f_0$ be the map that sends all the elements of $U_{\pi}$ to the least element of $U_{\tau}$ and all elements of $L_{\pi}$ to the least element of $L_{\tau}$. 

\begin{algorithm}
\caption{Find a minimum embedding of $\pi$ into $\tau$, or demonstrate that no embeddings exist.}
\label{alg-rigid-case}
\begin{algorithmic}
\STATE{Initialise: $f \leftarrow f_0$.}
\STATE{Compute: $P(f)$.}
\WHILE{$f$ is defined everywhere, and $P(f)$ is non-empty}
\STATE{Choose $x \in P(f)$.}
\STATE{Update: $f(x) \leftarrow \max \{ f(x^{\la})^{\ra}_{\type{x}}, f(x^{\da})^{\ua}_{\type{x}} \} $}
\STATE{Recompute: $P(f)$.}
\ENDWHILE
\STATE{\textbf{Return}: $f$, which, if everywhere defined, equals $e_{\min}$.}
\end{algorithmic}
\end{algorithm}

The correctness of this algorithm is easy to establish. Within the while loop, if $f$ is everywhere defined, $P(f)$ is non-empty, and $x$ is chosen for the update step, then the updated version of $f$ is strictly greater than the original at $x$, and has the same value elsewhere. Since the set of rigid maps is finite, the loop can be executed a bounded number of times, and the algorithm halts.  In the case where $e_{\min}$ exists, we certainly have $f_0 \leq e_{\min}$. So, by Proposition \ref{prop-allowed-update}, it is always the case that $f \leq e_{\min}$. Therefore, when the loop terminates, the algorithm returns an embedding that is less than or equal to $e_{\min}$, and hence must equal $e_{\min}$. Should $e_{\min}$ not exist, then termination can only occur because $f$ is not everywhere defined, and so the algorithm fails as required in this case.

We can further combine the correctness analysis with a run-time analysis to obtain the following.

\begin{proposition}
Given a rigid $321$-avoiding permutation $\pi$ of size $k$ and a $321$-avoiding permutation $\tau$ of size $n$ there is an algorithm which determines an embedding of $\pi$ into $\tau$ if one exists, and fails otherwise, whose run-time is $O(kn)$.
\end{proposition}

\begin{proof}
The algorithm in question is Algorithm \ref{alg-rigid-case}, and it remains to show that we can achieve the bound claimed for the run-time. As noted, each execution of the loop increases the value of $f(x)$ for at least one $x$ (in the linear ordering, $\leq$, of either $U_{\tau}$ or $L_{\tau}$). Since there are at most $n$ possible values any $f(x)$ can take, and only $k$ distinct $x$, the loop certainly executes not more than $kn$ times. So, if we can establish that the computation in the loop can be carried out in constant time, the claim follows.

In an initialisation phase (not part of the algorithm proper) we can certainly compute tables of all the values $x^{a}_{b}$ for $x$ in both $\pi$ and $\tau$, $a \in \{ \la, \ra, \da, \ua \}$ and $b$ either absent or equal to one of $L$ or $U$. For $\pi$ this can be done in $O(k)$ time, and for $\tau$ in $O(n)$ time, so this initialisation can be absorbed into the claimed run-time. This ensures that  the ``Update'' operations in the loop can be carried out in constant time. We can maintain $P(f)$ as a queue, and separately maintain an array of boolean values that indicate whether or not $x \in P(f)$. To start the loop, we dequeue some $x$. The update operation ensures that $x$ is no longer a problem, so we can set its value in the array to \texttt{false}. Moreover, the update operation only changes the value of $f(x)$, and increases it. So it cannot ``solve'' any existing problem (other than that of $x$) and the only other way that it could change the problem set would be if $f(x)$ moved to the right of $f(x^{\ra})$ or above $f(x^{\ua})$. Therefore, in the recompute phase we only need to check those two possibilities, and enqueue $x^{\ra}$ and/or $x^{\ua}$ (setting their boolean values in the array to \texttt{true}) if necessary. By making reference to the array, we can ensure that we never have duplicate elements in the queue -- so every iteration of the loop really does result in a proper update.
\end{proof}

Let us end this section by providing a simple example illustrating how the presented algorithm works.

\begin{example}\label{ex:algo1}
Let us consider the text permutation $\tau=3 \, 1 \, 2 \, 4 \, 5 \, 9 \, 6 \, 7 \, 10 \, 8 \, 11 \, 13 \, 12$ represented in Figure~\ref{fig-321} and the pattern $\pi= 2 \, 1 \, 4 \, 5 \, 3$. Note that $\pi$ is indeed rigid, whereas $\tau$ is not; we can however ignore the fluid elements when looking for an embedding of $\pi$ into $\tau$ as explained above. The upper elements in $\pi$ are $2, 4$ and $5$ and the lower elements are $1$ and $3$. We now describe a possible run of the algorithm (the order in which problems are resolved is not determined):
\begin{enumerate}
\item We start with the initial rigid mapping $f=f_0$ defined as follows: $f_0(1)=f_0(3)=1$ and $f_0(2)=f_0(4)=f_0(5)=3$. By checking the condition in equation~\eqref{eqn:problem-condition} we see that all elements except $1$ and $2$ are problems: $P(f_0)=P_0= \left\lbrace 3, 4, 5\right\rbrace $.
\item We resolve the problem $x=4$ for which we have $\max \{ f(x^{\la})^{\ra}_{U}, f(x^{\da})^{\ua}_{U} \}=9$ and update $f$ such that $f(4)=9$. In order to recompute $P(f)$, we only need to check $x^{\ra}=x^{\ua}=5$. We cannot possibly have resolved the problem $5$ at the same time, so it remains in $P(f)$ and we have $P(f)= \left\lbrace 3, 5\right\rbrace $.
\item We resolve the problem $x=5$ for which we have $\max \{ f(x^{\la})^{\ra}_{U}, f(x^{\da})^{\ua}_{U} \}=10$ and update $f$ such that $f(5)=10$. In order to recompute $P(f)$, we only need to check $x^{\ra}=3$ ($5^{\ua}$ is not defined). We cannot possibly have resolved the problem $3$ at the same time, so it remains in $P(f)$ and we have $P(f)= \left\lbrace 3 \right\rbrace $.
\item We resolve the last problem $x=3$ for which we have $\max \{ f(x^{\la})^{\ra}_{L}, f(x^{\da})^{\ua}_{L} \}=8$ and update $f$ such that $f(3)=8$. 
In order to recompute $P(f)$, we only need to check $x^{\ua}=4$ ($3^{\ra}$ is not defined). The element $4$ is no longer a problem since it is large enough and thus $P(f)$ is empty. 
\item The algorithm terminates successfully since $P(f)$ is empty and has found the minimal embedding $e=e_{\min}$ of $\pi$ into $\tau$ defined as follows: $e(2)=3$, $e(1)=1$, $e(4)=9$, $e(5)=10$ and $e(3)=8$.
\end{enumerate} 
\end{example}

\section{Fluid Elements and the $O(kn)$ Algorithm for $321$-Avoiding Permutations}
\label{sec-fluid}

In this section we aim to complete the proof of Theorem \ref{thm-complexity-321} and to do so we must face the issue of fluid elements in the pattern $\pi$. Since a fluid element participates in no $21$, each other element of $\pi$ is either below and left of it, or above and right of it. This is represented most easily using another notational convention. Suppose that $\sigma$ and $\theta$ are two permutations of size $m$ and $n$ respectively. Then $\sigma \oplus \theta $ is the permutation whose points are:
\[
\sigma \cup \{ (i + m, \theta(i) + m) \st i \in [n] \}.
\]
Informally, to form $\sigma \oplus \theta$ we just place $\theta$ above and to the right of $\sigma$. Clearly $\oplus$ is associative, though  of course not commutative. 

For any $321$-avoiding permutation $\pi$ there is a unique decomposition:
\[
\pi = \pi_1 \oplus \pi_2 \oplus \cdots \oplus \pi_t
\]
where, for $1 \leq i \leq t$, $\pi_i$ is either rigid or a singleton, and it is never the case that both $\pi_i$ and $\pi_{i+1}$ are rigid. The singleton elements of this representation correspond precisely to the fluid elements of $\pi$.
For an example, consider again Figure~\ref{fig-321} where the black squares correspond to the blocks $\pi_i$ of this representation.

Given $\pi$ of size $k$ we can easily compute this representation in $O(k)$ time, simply by finding the fluid elements of $\pi$ (which are those elements that are both left-to-right maxima and right-to-left minima). Henceforth, we assume that this representation is given.

In the algorithm to determine whether $\pi$ embeds in $\tau$ we will construct, for each $1 \leq i \leq t$ at most two embeddings of $\pi_1 \oplus \cdots \oplus \pi_i$ into $\tau$ in such a way that, if any embedding of $\pi$ into $\tau$ exists, then at least one of the two partial embeddings can be extended to a full embedding.

So we first consider the following question: given an embedding, $e_i$, of $\pi_1 \oplus \cdots \oplus \pi_i$ into $\tau$ that extends to an (unknown) embedding, $e$, of $\pi$ into $\tau$, how can we construct a pair of  embeddings of $\pi_1 \oplus \cdots \oplus \pi_i \oplus \pi_{i+1}$ into $\tau$, at least one of which extends to an embedding of $\pi$ into $\tau$? 

We distinguish three cases for $\pi_{i+1}$.
For this purpose, let $T_i$ denote the set of elements that lie above and to the right of the image of $e_i$. Then, the image of $e$ restricted to the elements corresponding to $\pi_{i+1}$ is contained in $T_i$. 
Let us first consider the case where $\pi_{i+1}$ is rigid.
Then the image of $e$ on the elements corresponding to $\pi_{i+1}$ must be greater than or equal to (point by point), the image of $\pi_{i+1}$ under its minimum embedding into $T_i$. 
Thus, if we choose the minimal embedding of $\pi_{i+1}$ into $T_i$, the resulting embedding $e_{i+1}$ extends to an embedding of $\pi$ into $\tau$.
Though $T_i$ is, strictly speaking, not a permutation all of its associated operators are the same as those of $\tau$ (except some are undefined, e.g., the leftmost element of $T_i$ has no left neighbour in $T_i$ but may well have one in $\tau$). So, in this case we can use Algorithm~\ref{alg-rigid-case} in order to find the minimal embedding of $\pi_{i+1}$ into $T_i$ and hereby obtain a single extension of $e_i$ with the required property.

A similarly easy case is where $\pi_{i+1}$ is a singleton, i.e., a fluid element and $T_i$ begins with its least element (which is a fluid element as well). Then nothing can be lost by mapping $\pi_{i+1}$ to that element. 

The only remaining case is where $\pi_{i+1}$ is a singleton and the first element of $T_i$ is not its minimum. Since every element of $T_i$ lies above its first element, or above and to the right of its minimum, we can extend $e_i$ in two ways -- one sending $\pi_{i+1}$ to the leftmost element of $T_i$ (which is an upper element) and one to its minimum (which is a lower element), and one of these must be extensible.

Now it seems that we might have a problem -- given two partial embeddings of $\pi_1 \oplus \cdots \oplus \pi_i$ might they not extend to three or four candidate embeddings of $\pi_1 \oplus \cdots \oplus \pi_i \oplus \pi_{i+1}$? Indeed this is the case, but only if $\pi_{i+1}$ is a singleton. If it has four possible images, two belong to $U_\tau$ and two to $L_\tau$. Since all further elements of $\pi$ lie above and to the right of this fluid element, we only need to retain the embeddings where its image is the lesser of the two in each of these sets. Likewise, if it has three possible images (one of which might be fluid), at least one of them can be ignored. Another way to say this is that because $\pi_1 \oplus \cdots \oplus \pi_i \oplus \pi_{i+1}$ ends with its maximum element, so do its images under the embedding. Among three or more elements of a $321$-avoiding permutation there are at least one and at most two elements that do not participate as the $2$ in a $12$ pattern. We need retain only those embeddings whose maximum is not such a $2$, as otherwise they could be replaced by an embedding with a smaller maximum in forming a full extension.

Since the sum of the size of the rigid permutations in the representation of $\pi$ is at most the total size of $\pi$, the parts of the algorithm where we construct minimal rigid embeddings still require at most $O(kn)$ time in total. Dealing with singletons (fluid elements) clearly requires only constant time since we can find the next (to the right) fluid/upper/lower element in $\tau$ in constant time. Also, filtering out non-optimal extensions can be done in constant time since only the maximal elements of these extensions have to be compared and at most four extensions exist at the same time.
We conclude that the total cost of the algorithm is still $O(kn)$.
If $\tau$ contains $\pi$ the algorithm terminates successfully and returns one or possibly two embeddings.

What if no embedding exists? Then, following the plan above as if it did (beginning from an empty map, i.e., the case $i = 0$) we must at some point reach a failing case of Algorithm \ref{alg-rigid-case}, or possibly encounter an empty $T_i$. In either case, we fail since we have demonstrated that no embedding can be possible.

This completes the proof of Theorem \ref{thm-complexity-321}.

Again, let us provide an example demonstrating how the algorithm for arbitrary $321$-avoiding patterns works.

\begin{example}
As in Example~\ref{ex:algo1}, we consider the text permutation $\tau=3 \, 1 \, 2 \, 4 \, 5 \, 9 \, 6 \, 7 \, 10 \, 8 \, 11 \, 13 \, 12$ represented in Figure~\ref{fig-321}. The pattern is $\pi= 2 \, 1 \, 3 \, 4 \, 5 \, 7 \, 6 \, 8$. The upper elements in $\pi$ are $2$ and $7$, the lower ones are $1$ and $6$, and the fluid elements are $3$, $4$, $5$ and $8$. The algorithm proceeds block by block in the decomposition of $\pi$.
\begin{enumerate}
\item It starts with the rigid block consisting of the elements $2$ and $1$. Algorithm~\ref{alg-rigid-case} takes care of this block and, as in Example~\ref{ex:algo1}, $e(2)=3$ and $e(1)=1$.
\item The next block $\pi_2$ is the singleton element $3$. $T_1$, the set of elements that lie above and to the right of the image of $e$ starts with a fluid element and thus we can set $e(3)=4$.
\item We have the same situation for $\pi_3$ which consists of the singleton element $4$ and set $e(4)=5$.
\item The block $\pi_4$ is again a singleton element. However, $T_3$ does not start with its minimal element and thus two choices are possible for $e(5)$: we can either send $5$ to the leftmost upper element in $T_3$ or to the leftmost lower element. We store these two possibilities: $e_U(5)=9$ and $e_L(5)=6$.
\item The next block $\pi_5$ is rigid and we thus apply Algorithm~\ref{alg-rigid-case}  which is not detailed here. For the choice $e_U(5)=9$ it leads to $e_U(7)=13$ and $e_U(6)=12$ whereas for $e_L(5)=6$ it leads to $e_L(7)=10$ and $e_L(6)=8$. These two partial embeddings are rigid and thus comparable: $e_L \leq  e_U$ and we can disregard $e_U$. This is a good choice, since $e_U$ cannot be extended to an embedding of $\pi$ into $\tau$ since the last element $8$ cannot be mapped anywhere.
\item It remains to determine $e(8)$. Since $T_5$ starts with its minimal element we can choose this one and set $e(8)=11$.
\item The algorithm terminates successfully and returns an embedding of $\pi$ into $\tau$:  $e(2)=3$, $e(1)=1$, $e(3)=4$, $e(4)=5$, $e(5)=6$, $e(7)=10$, $e(6)=8$ and $e(8)=11$.
\end{enumerate}
\end{example}

\section{Skew-Merged Permutations}
\label{sec-skew-merged}

The permutations avoiding $321$ can be partitioned into two monotone increasing sequences. Of course the permutations avoiding $123$ can be similarly partitioned (into decreasing sequences) and the results of the previous section apply to them as well. However, the class of \emph{skew-merged permutations}, those that can be partitioned into an increasing and a decreasing sequence, requires further analysis, though as we shall see the analogue of Theorem \ref{thm-complexity-321} is also true in this context.
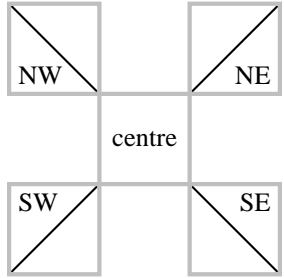
\begin{figure}
\begin{center}
	\begin{tikzpicture}[scale=0.03, baseline=(current bounding box.center)]
		\draw [thick, line cap=round] (-60,60)--(60,-60);
		\draw [thick, line cap=round] (-60,-60)--(60,60);
		\draw [lightgray, ultra thick, line cap=round] (-20,20) rectangle (-60,60);
		\draw [lightgray, ultra thick, line cap=round] (-20,-20) rectangle (-60,-60);
		\draw [lightgray, ultra thick, line cap=round] (20,20) rectangle (60,60);
		\draw [lightgray, ultra thick, line cap=round] (20,-20) rectangle (60,-60);
		\draw [lightgray, ultra thick, line cap=round, fill=white] (-20,-20) rectangle (20,20);
		\node at (60,20) [above left] {NE};
		\node at (-60,20) [above right] {NW};
		\node at (-60,-20) [below right] {SW};
		\node at (60,-20) [below left] {SE};
		\node at (0,0) {centre};
	\end{tikzpicture}
\end{center}
\caption{The decomposition of a skew-merged permutation into its centre and four corners.}
\label{fig-skew-merged}
\end{figure}

Towards this goal, we first identify a set of \emph{rigid} elements of a skew-merged permutation. In Figure \ref{fig-skew-merged} these are the elements lying in the corner regions. Specifically we say that an element of a skew-merged permutation is of type: \\
\centerline{
\begin{tabular}{ll}
NE & if it participates as a $3$ in a $213$; \\
NW & if it participates as a $3$ in a $312$; \\
SW & if it participates as a $1$ in a $132$; \\
SE & if it participates as a $1$ in a $231$,
\end{tabular}}
and we call any other element of a skew-merged permutation \emph{central}. We first verify that the illustration of a skew-merged permutation shown in Figure \ref{fig-skew-merged} is correct. This is a result due to Atkinson~\cite{atkinson:permutations-wh:}, and so we only sketch part of the proof to give its flavour.

\begin{proposition}
\label{prop-skew-merged-structure}
The elements of a skew-merged permutation decompose by type as shown in Figure \ref{fig-skew-merged}. Moreover, the central elements form a monotone subsequence.
\end{proposition}

\begin{proof}
Recall that another characterisation of skew-merged permutations is the following: they are those permutations that do not contain either $3412$ or $2143$. 

Let a skew-merged permutation $\pi$ be given, and suppose that $\pi = I \cup D$ is a partition of $\pi$ into a monotone increasing and monotone decreasing sequence. Consider first elements of type NE (all other types can be handled by parallel arguments due to symmetry). Since any such participates as a  $3$ in a $213$, it must belong to $I$ (otherwise, the elements participating as the $2$ and $1$ would both belong to $I$ which is of course impossible). So the elements of type NE form a monotone increasing sequence. 

Suppose that $C$ is of type NE, with $BAC$ an occurrence of $213$ and $a$ is of type SW with $acb$ an occurrence of $132$. Then $a \in I$ for similar reasons to the preceding ones. If $C$ preceded $a$ (and hence was also smaller than it) then, $BAcb$ would be an occurrence of $2143$. So, all elements of type SW lie below and to the left of those of type NE. Now suppose that $z$ is of type NW, with $zxy$ an occurrence of $312$. If $C$ were to precede $z$ we would have various cases: first if $C$ lay below $y$ then $BAzy$ would be $2143$, if $C$ lay above $y$ but below $z$ then $Czxy$ would be $3412$, if $C$ lay above $z$ and $B$ above $y$, then $BCxy$ would be $3412$, but if $B$ lay below $y$ then $BAzy$ would be $2143$. As all these cases lead to contradictions, $C$ must follow $z$. 

All other cases can be dealt with similarly. Finally, to see that the central elements form a monotone sequence observe that they must certainly avoid all of $132$, $213$, $231$, and $312$ lest some of them be non-central. But, only monotone permutations (of either type) avoid these four permutations.
\end{proof}

This decomposition can be computed in linear time:

\begin{lemma}
Given a skew-merged permutation of size $n$, there is an algorithm that computes its partition into types in $O(n)$-time.
\end{lemma}

\begin{proof}
Let $\theta$ be an arbitrary skew-merged permutation. 
Notice that the part of $\theta$ to the left of the leftmost element of type NE or SE avoids $231$ and $213$. Such permutations have a characteristic $>$ shape since any element must not be intermediate in value between two to its right. We are interested in finding the maximum prefix of $\theta$ which has this characteristic shape, or what amounts to the same thing, the leftmost element of $\theta$ such that the prefix ending at that element involves $231$ or $213$. 

This can be accomplished in linear time: we scan $\theta$ from left to right and determine for every position $i$ whether it is an ascent ($\theta(i)<\theta(i+1)$) or a descent ($\theta(i)>\theta(i+1)$). At any moment we only store the last encountered ascent $a$ and descent $d$. The element $\theta(i)$ plays the role of a $1$ in a $231$ pattern, if $\theta(i)<\theta(a)$; it plays the role of a $3$ in a $213$, if $\theta(i)>\theta(d)$.
If either of the two conditions apply to position $i$, we have identified the leftmost element of type NE or SE. 
That is, we have found the boundary line between the centre region and the Eastern region of $\theta$.

In a similar manner we can find all of the boundary lines: by scanning $\theta$ from right to left we find the boundary between West and centre, by scanning from bottom to top we find the boundary between South and centre and by scanning from top to bottom we find the boundary between North and centre.
We can thus compute the partition of $\theta$ into types by scanning $\theta$ four times.
\end{proof}

We will now describe an algorithm for skew-merged patterns and texts and provide the necessary theoretic background. This algorithm consists of two main parts: In the first part, the non central elements of the pattern $\pi$ are embedded into $\tau$ using a similar approach as for rigid permutations and adapting Algorithm~\ref{alg-rigid-case} which will deliver a minimal embedding of the non-central elements. In the second part, we will extend this minimal embedding to the central elements of $\pi$.

In this sense, the non-central elements of a skew-merged permutation correspond to the rigid elements of a $321$-avoiding permutation. Since they are defined by the occurrence of certain patterns and since embeddings preserve such patterns it is immediately clear that if $e: \pi \to \tau$ is an embedding of one skew-merged permutation into another, then $e$ must preserve the type of all non-central elements. 

In order to be able to speak of minimal embeddings in the context of skew-merged permutations, we need to introduce some new notation. For two non-central elements of the same type we write $x \lhd y$ if $x$ lies strictly further out from the center than $y$ ($x \unlhd y$ will mean that either $x \lhd y$ or $x=y$). The minimum with respect to this relation $\lhd$ is denoted by $\textit{outer}$ and the maximum by $\textit{inner}$. 
For two embeddings, $e_1$ and $e_2$ of the non-central elements of $\pi$ into $\tau$ that preserve types we can define their meet by
\[
e_1 \wedge e_2 (x) = \textit{outer} \{e_1(x), e_2(x)\}
\]
for all non-central $x \in \pi$. Then, just as in the $321$-avoiding case, $e_1 \wedge e_2$ is also an embedding of the non-central elements of $\pi$ into those of $\tau$:

\begin{lemma}
Let $\pi$ be a skew-merged pattern with no central elements and $\tau$ be an arbitrary skew-merged permutation. Then the following holds: If $e_1$ and $e_2$ are embeddings of $\pi$ into $\tau$ then their meet $f := e_1 \wedge e_2$ as defined above is an embedding as well.
\end{lemma}

\begin{proof}
Let $x \neq y$ be two elements in $\pi$ and let us assume that $x$ lies to the left of $y$ in $\pi$. We need to show that $f(x)$ lies to the left of  $f(y)$ in $\tau$ and that the horizontal relation between $x$ and $y$ is preserved as well. The key argument is that taking the minimum of the elements in the above sense automatically translates into taking their actual minimum or maximum (equivalently, the leftmost or rightmost element), depending on the type of element. In order to give a formal proof, we distinguish between three cases.
\begin{itemize}
\item If $x$ and $y$ are of the same type. We detail the case of SW elements here, as the other ones are analogous (one simply needs to replace ``minimum'' by ``maximum'' and/or ``left of'' by ``right of'' depending on the type). In this case, taking the minimum of the elements in the sense defined earlier is nothing else than taking their actual minimum, which again is the same as taking the left-most element. Since we have that $f(x) \leq e_1(x) < e_1(y)$ (and $f(x)$ is to the left of $e_1(y)$) as well as $f(x) \leq e_2(x) < e_2(y)$ (and $f(x)$ is to the left of $e_2(y)$), it follows that $f(x) < \textit{outer}(e_1(y),e_2(y))=f(y)$ (and $f(x)$ is to the left of $f(y)$).
\item If $x$ and $y$ lie in opposite corners of the diagram. In this case the statement follows immediately from the fact that an embedding preserves types. Indeed, all SW elements are to the left of and smaller than NE ones and all NW elements are to the left of and larger than SE ones. Thus both the vertical as well as the horizontal relation between $x$ and $y$ is preserved.
\item The remaining cases, where $x$ and $y$ are not of the same type, but are both elements in the south, north, east or west. We detail the case of two elements in the north, i.e., $x$ is a NW element and $y$ a NE one. The other cases can be dealt with analogously (by interchanging minimum with maximum or vertical with horizontal positions). Without loss of generality, we further assume that $x < y$. First, it is clear that $f(x)$ lies strictly to the left of $f(y)$ since types are preserved. Second, regarding the horizontal relation between $x$ and $y$, let us note that taking the element that is furthest away from the centre translates into taking the maximum. Thus, we have that $f(y) \geq e_1(y) > e_1(x)$ as well as $f(y) \geq e_2(y) > e_2(x)$ which implies that $f(y) > f(x)$.
\end{itemize}
The consideration of these cases completes the proof.
\end{proof}

Observe the following: if either $e_1$ or $e_2$ was the restriction of an actual embedding, $e$, of $\pi$ into $\tau$ to the non-central elements then we can extend the mapping $e_1 \wedge e_2$ to central elements using $e$ there, and thereby obtain an embedding. So, among all embeddings of $\pi$ into $\tau$ there is one whose effect on the non-central elements is the minimum of all the embeddings of the non-central elements of $\pi$ into those of $\tau$. We will see later on how such an extension to the central elements of $\pi$ can be found.

This minimum embedding of the non-central elements can be found by modifying the definition of the problem set and the update rule of Algorithm \ref{alg-rigid-case}. The only thing we need to do in order to reflect the new notion of minimum/maximum in this definition, is to redefine the notation introduced in the Introduction. Given a non-central element $x$ in a skew-merged permutation $\pi$, we denote by (wherever possible): \\
\begin{tabular}{ll}
$x^{oh}$ & the next non-central element further \underline{o}ut from the center in  \underline{h}orizontal direction, \\
$x^{ih}$ & the next non-central element towards the center (\underline{i}nwards) in  \underline{h}orizontal direction, \\
$x^{iv}$ & the next non-central element towards the center (\underline{i}nwards) in  \underline{v}ertical direction, and \\
$x^{ov}$ & the next non-central element further \underline{o}ut from the center in  \underline{v}ertical direction.
\end{tabular} \\
For example, in the skew-merged pattern $\pi$ depicted in Figure~\ref{fig-skew-merged-ex} and $x= 7$, we have: $x^{oh}=1$, $x^{ih}=\bot$, $x^{iv}=6$, and $x^{ov}=\bot$.

We also define the type of a non-central element $x$ in a skew-merged permutation, $T(x)$, to be the corner in which $x$ lies, i.e., $T(x)$ can be NW, SW, NE or SE.
Moreover, we extend the notation introduced above as follows: For $a\in\{oh,ih,iv,ov\}$ and $b\in\{NE,SE,SW,NW\}$, we define $x^{a}_{b}$ to be the next non-central element in $\pi$ according to direction $a$ that is of type $b$. In other words, $x^a_b$ is the first element in the sequence $(x^a, (x^a)^a, \dots)$ of type $b$.
If there is no such element, i.e., no element in $(x^a, (x^a)^a, \dots)$ is of type $b$, then we set $x^{a}_{b}=\bot$.
For example, in the skew-merged pattern $\pi$ depicted in Figure~\ref{fig-skew-merged-ex} and  $x=5$, we have $x^{ov}_{NW}=7$, $x^{ov}_{NE}=6$, whereas  $x^{oh}_{SE}=\bot$ and $x^{iv}_{SE}=\bot$.

With this new notation, one can see that an analogue of Proposition~\ref{prop-allowed-update} holds for skew-merged permutations:

\begin{proposition}
Suppose that $e$ is an embedding of the non-central elements of $\pi$ into $\tau$, $f$ is a mapping of the non-central elements of $\pi$ into $\tau$ that preserves types, and, for all non-central $x \in \pi$, $f(x) \unlhd e(x)$. Then, for all non-central $x \in \pi$:
\[
\textit{inner} \left\lbrace  f(x^{oh})^{ih}_{\type{x}}, f(x^{ov})^{iv}_{\type{x}} \right\rbrace  \unlhd e(x).
\] 
\end{proposition}

The proof of this Proposition is analogous to the one of Proposition~\ref{prop-allowed-update}. 

We thus say that a non-central element $x$ of a skew-merged permutation is a \textit{problem} if:
\begin{equation}\label{eqn:problem-condition-skew}
f(x) \lhd \textit{inner} \left\lbrace  f(x^{oh})^{ih}_{\type{x}}, f(x^{ov})^{iv}_{\type{x}} \right\rbrace,
\end{equation}
for a mapping $f$ of the non-central elements of $\pi$ into $\tau$ that preserves types. Moreover, when we resolve the problem $x$ by updating the value of $f(x)$ this is done analogously to the case of $321$-avoiding permutations and we set $f(x) = \textit{inner} \left\lbrace  f(x^{oh})^{ih}_{\type{x}}, f(x^{ov})^{iv}_{\type{x}} \right\rbrace$.

This finishes the description of the necessary modifications of Algorithm~\ref{alg-rigid-case}.
As for Algorithm~\ref{alg-rigid-case} we assume that $x^a$ and  $x^a_b$ for $a\in\{oh,ih,iv,ov\}$ and $b\in\{NE,SE,SW,NW\}$ is precomputed and thus can be found in constant time.
Given the decomposition of $\pi$ and $\tau$ into types, these precomputations can be done in linear time.
Both the steps required for the update of $f$ and the recomputation of the problem set $P(f)$ can be carried out in constant time.

To complete the proof of Theorem \ref{thm-complexity-skew-merged} we must show that, having found a minimum embedding of the non-central elements of $\pi$ to those of $\tau$, the existence of a full embedding can also be determined sufficiently quickly.
%
We need to determine whether or not the central part of $\pi$ can be embedded into the remainder of $\tau$, i.e., the set  of elements in $\tau$ which consists of central elements and all adjacent elements that have not yet been used in the minimum embedding.
The central part of $\pi$ is a monotone pattern of a certain size at most $k$, and the remaining part of $\tau$ is a skew-merged permutation of size at most $n$ (whose endpoints we know). 

In general, finding a longest increasing (or decreasing) subsequence of size $k$ in a permutation of size $n$ can be done in time $O(n \log \log k)$~\cite{crochemore:fast-computatio:}. Thus, checking whether the central part can be embedded into the remaining part of $\tau$ can be done within the $O(n k)$ runtime bound of our algorithm. In the special case of skew-merged permutations, finding a longest increasing (resp., decreasing) subsequence can even be done in $O(n)$ time. To be more precise, $O(n)$ time is only required for obtaining the partition into five types as represented in Figure~\ref{fig-skew-merged} (which is already available in our case); the remaining steps require only constant time.

Indeed, for longest increasing  subsequences the following observations can be made (the case of  decreasing subsequences can be treated analogously): The elements of type SW and NE will always contribute to a longest increasing subsequence. Moreover, such a subsequence also contains as many elements as possible from the centre, i.e., if the center is increasing then all elements contribute to a longest increasing subsequence and if the center is decreasing we can arbitrarily pick one centre element. Note that it is never advantageous to include elements of type NW or SE. This can be seen as follows: At most one NW or SE element can be part of an increasing subsequence. Thus, if the centre is non-empty, it is certainly not advantageous to include a NW or SE element. Let us assume that the centre is empty. An element of type NW occurs as a $3$ in a $312$ pattern.  Among the elements playing the role of the $1$ and the $2$, at least one element (and possibly both of them) is of type SW or NE. Thus, including an element of type NW would force us to exclude one or two elements of type SW or NE. In other words, we cannot increase the size of an increasing subsequence by adding an element of type NW. A similar argument holds for elements of type SE.  We conclude that for the size of the longest increasing subsequence we only have to add the number of elements of type SW and NE as well as the size of the longest increasing subsequence in the central part.
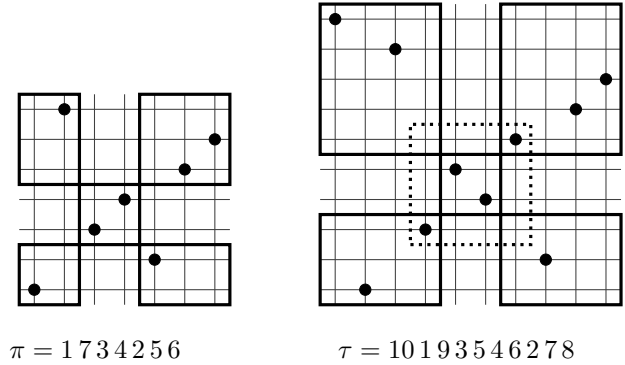
\begin{figure}
\begin{center}
\begin{tikzpicture}[scale=0.4]
\tikzstyle{l}=[circle, draw, fill=black, inner sep=0.0525cm]

	\begin{scope}
		\foreach \i in {1,...,7} {
		\draw [darkgray] (\i,0.5) -- (\i,7.5);
	}
	\foreach \i in {1,...,7} {
		\draw [darkgray] (0.5,\i) -- (7.5,\i);
	}

	\node[l] at (1,1) {};
	\node[l] at (2,7) {};
	\node[l] at (3,3) {};
	\node[l] at (4,4) {};
	\node[l] at (5,2) {};
	\node[l] at (6,5) {};
	\node[l] at (7,6) {};

	\draw[draw=black, very thick] (0.5,0.5) rectangle (2.5,2.5);
	\draw[draw=black, very thick] (7.5,0.5) rectangle (4.5,2.5);
	\draw[draw=black, very thick] (0.5,7.5) rectangle (2.5,4.5);
	\draw[draw=black, very thick] (7.5,7.5) rectangle (4.5,4.5);
	\draw[draw=black, very thick] (2.5,2.5) rectangle (4.5,4.5);
	\node at (3,-1) {$\pi = 1 \, 7 \, 3 \, 4 \, 2 \, 5  \, 6$};
	\end{scope}
	
	\begin{scope}[shift={(10,0)}]
	\foreach \i in {1,...,10} {
		\draw [darkgray] (\i,0.5) -- (\i,10.5);
	}
	\foreach \i in {1,...,10} {
		\draw [darkgray] (0.5,\i) -- (10.5,\i);
	}

	\node[l] at (1,10) {};
	\node[l] at (2,1) {};
	\node[l] at (3,9) {};
	\node[l] at (4,3) {};
	\node[l] at (5,5) {};
	\node[l] at (6,4) {};
	\node[l] at (7,6) {};
	\node[l] at (8,2) {};
	\node[l] at (9,7) {};
	\node[l] at (10,8) {};
	\node at (5,-1) {$\tau = 10 \, 1 \, 9 \, 3 \, 5 \, 4 \,  6 \, 2 \, 7 \, 8$};

	\draw[draw=black, very thick] (0.5,0.5) rectangle (4.5,3.5);
	\draw[draw=black, very thick] (10.5,0.5) rectangle (6.5,3.5);
	\draw[draw=black, very thick] (0.5,10.5) rectangle (4.5,5.5);
	\draw[draw=black, very thick] (10.5,10.5) rectangle (6.5,5.5);
	\draw[draw=black, very thick] (4.5,3.5) rectangle (6.5,5.5);
	\draw[draw=black, very thick, dotted] (3.5,2.5) rectangle (7.5,6.5);
	\end{scope}
\end{tikzpicture}
\end{center}
\caption{Decomposition of the skew-merged permutations $\pi$ and $\tau$ into their centres and four corners.}
\label{fig-skew-merged-ex}
\end{figure}
Let us end this section by providing a simple example illustrating how this modified version of Algorithm~\ref{alg-rigid-case} works.

\begin{example}
Let us consider the text permutation $\tau=10 \, 1 \, 9 \, 3 \, 5 \, 4  \, 6 \, 2 \,  7 \, 8$ and the pattern $\pi= 1 \, 7 \, 3 \, 4 \, 2 \, 5  \, 6$. Both permutations and their decomposition into types are shown in Figure~\ref{fig-skew-merged-ex}. 
We start by describing a possible run of the algorithm (the order in which problems are resolved is not determined) finding the minimal embedding of the non-central elements of $\pi$ into $\tau$:
\begin{enumerate}
\item We start with the initial mapping $f=f_0$ that sends all non-central elements of one type in $\pi$ to the minimal element of this type in $\tau$ (i.e., the element that is furthest out from the center). It is defined as follows:
$f(1)=1$, $f(7)=10$, $f(5)=f(6)=8$ and $f(2)=2$.
We compute the problem set using the condition in equation~\eqref{eqn:problem-condition-skew} and obtain $P(f)= \left\lbrace 5,7\right\rbrace$.
\item We resolve the problem $x=5$  for which we have $\textit{inner} \left\lbrace  f(x^{oh})^{ih}_{NE}, f(x^{ov})^{iv}_{NE} \right\rbrace=7$ and update $f$ such that $f(5)=7$. In order to recompute $P(f)$, we only need to check $x^{ih}=2$ since $x^{iv}$ is not defined. The choice $f(2)=2$ does not create a problem with this  new choice for $f(5)$.
We cannot possibly have resolved the problem $7$ at the same time, so it remains in $P(f)$ and we have $P(f)= \left\lbrace 7\right\rbrace $.
\item We resolve the problem $x=7$ for which we have $\textit{inner} \left\lbrace  f(x^{oh})^{ih}_{NW}, f(x^{ov})^{iv}_{NW} \right\rbrace=9$ and update $f$ such that $f(7)=9$. In order to recompute $P(f)$, we only need to check $x^{iv}=6$ since $x^{ih}$ is not defined. The choice $f(6)=8$ does not create a problem with this  new choice for $f(7)$.
\item The algorithm has found the minimal embedding $e=e_{\min}$ of the non-central elements of $\pi$ into $\tau$ defined as follows: 
$f(1)=1$, $f(7)=9$, $f(2)=2$, $f(5)=7$ and $f(6)=8$.
\item We need to map the central elements $3$ and $4$ of $\pi$ into the remaining part of $\tau$ (marked by a dotted line in Figure~\ref{fig-skew-merged-ex}). Since the central elements of $\pi$ consist of an increasing subsequence of size two, we can choose any such subsequence within the dotted area in $\tau$. We decide to set $f(3)=3$ and $f(4)=5$ which finally gives an embedding of $\pi$ into $\tau$.
\end{enumerate} 
\end{example}

\section{Concluding Remarks}

We conclude by mentioning some open problems related to this work. We have seen in Theorem~\ref{thm-complexity-321} that \textsc{$\Av(321)$-PPM} can be solved in $O(kn)$ time. Guillemot and Vialette showed that the more general \textsc{$\Av(321)$-pattern PPM} problem can be solved in $\mathcal{O}(k n^{4 \sqrt{k}+12})$ time. It is an open problem whether \textsc{$\Av(321)$-pattern PPM} can be solved in polynomial time. Note that if the pattern avoids $132$, $231$, $213$ or $312$ then it is automatically separable and thus the \textsc{$\C$-PPM} problem and the \textsc{$\C$-pattern PPM} problem for all four classes $\Av(132)$, $\Av(231)$, $\Av(213)$ or $\Av(312)$ can be solved in polynomial time. Consequently the \textsc{$\Av(321)$-pattern PPM}---which is equivalent to the \textsc{$\Av(123)$-pattern PPM}---is the only open case for \textsc{$\Av(\beta)$-pattern PPM} where $\beta$ has size $3$.

In case \textsc{$\Av(321)$-pattern PPM} turns out to be \textsf{NP}-complete, \textsc{$\Av(\beta)$-pattern PPM} will also be \textsf{NP}-complete if $\beta$ is any permutation of size four other than $2143$, $3142$, $2413$, or $3412$. Interestingly, this list contains exactly those patterns that define the classes of skew-merged and of separable permutations. Moreover, \textsf{NP}-completeness of \textsc{$\Av(321)$-pattern PPM} would imply that \textsc{$\Av(\beta)$-pattern PPM} is \textsf{NP}-complete for $\beta$ of size five or more, since by Erd\H{o}s--Szekeres Theorem~\cite{erdos:a-combinatorial:} every permutation of size at least five contains $123$ or $321$.

Looking at the big picture, Theorems~\ref{thm-complexity-321} and \ref{thm-complexity-skew-merged} show that \textsc{$\C$-PPM} can be solved in polynomial time for $\Av(321)$ and $\Av(2143,3412)$, respectively. It might be that \textsc{$\C$-PPM} is always polynomial-time solvable for a fixed, proper class $\C$. It would be of considerable interest to either establish this statement or to prove a dichotomy theorem that distinguishes permutation classes for which \textsc{$\C$-PPM} is polynomial-time solvable and those that yield hard \textsc{$\C$-PPM} instances. The same question can be asked for \textsc{$\C$-pattern PPM}, although it seems rather unlikely that this problem is polynomial time solvable for every fixed, proper class $\C$.

\medskip
\noindent\textbf{Note added in proof.}
After a draft of this paper was posted on the arXiv, Jel{\'{\i}}nek and Kyn{\v{c}}l~\cite{jelinek:hardness-of-per:} established that the \textsc{$\Av(\beta)$-pattern PPM} problem is indeed \textsf{NP}-complete for every
\[
	\beta\notin\{1,12,21,132,213,231,312\}.
\]
They further showed that the \textsc{$\Av(4321)$-PPM} problem is \textsf{NP}-complete, even when the pattern is restricted to be $321$-avoiding.

\bibliographystyle{abbrv}
\bibliography{refs}

\def\cprime{$'$}
\begin{thebibliography}{10}

\bibitem{ahal:on-complexity-o:}
S.~Ahal and Y.~Rabinovich.
\newblock On complexity of the subpattern problem.
\newblock {\em SIAM J. Discrete Math.}, 22(2):629--649, 2008.

\bibitem{albert:algorithms-for-:}
M.~H. Albert, R.~E.~L. Aldred, M.~D. Atkinson, and D.~A. Holton.
\newblock Algorithms for pattern involvement in permutations.
\newblock In P.~Eades and T.~Takaoka, editors, {\em 12th International
  Symposium on Algorithms and Computation (ISAAC)}, volume 2223 of {\em Lecture
  Notes in Comput. Sci.}, pages 355--366. Springer, Berlin, Germany, 2001.

\bibitem{albert:growth-rates-fo:}
M.~H. Albert, M.~D. Atkinson, R.~Brignall, N.~Ru{\v{s}}kuc, R.~Smith, and
  J.~West.
\newblock Growth rates for subclasses of $\textrm{Av}(321)$.
\newblock {\em Electron. J. Combin.}, 17:Paper 141, 16 pp., 2010.

\bibitem{atkinson:permutations-wh:}
M.~D. Atkinson.
\newblock Permutations which are the union of an increasing and a decreasing
  subsequence.
\newblock {\em Electron. J. Combin.}, 5:Paper 6, 13 pp., 1998.

\bibitem{bose:pattern-matchin:}
P.~Bose, J.~F. Buss, and A.~Lubiw.
\newblock Pattern matching for permutations.
\newblock {\em Inform. Process. Lett.}, 65(5):277--283, 1998.

\bibitem{bruner:a-fast-algorith:}
M.-L. Bruner and M.~Lackner.
\newblock A fast algorithm for permutation pattern matching based on
  alternating runs.
\newblock {\em Algorithmica}, 75(1):84--117, 2016.

\bibitem{crochemore:fast-computatio:}
M.~Crochemore and E.~Porat.
\newblock Fast computation of a longest increasing subsequence and application.
\newblock {\em Inform. and Comput.}, 208(9):1054--1059, 2010.

\bibitem{erdos:a-combinatorial:}
P.~Erd{\H{o}}s and G.~Szekeres.
\newblock A combinatorial problem in geometry.
\newblock {\em Compos. Math.}, 2:463--470, 1935.

\bibitem{guillemot:finding-small-p:}
S.~Guillemot and D.~Marx.
\newblock Finding small patterns in permutations in linear time.
\newblock In {\em Proceedings of the Twenty-Fifth Annual ACM-SIAM Symposium on
  Discrete Algorithms (SODA)}, pages 82--101. SIAM, Philadelphia, Pennsylvania,
  2014.

\bibitem{guillemot:pattern-matchin:}
S.~Guillemot and S.~Vialette.
\newblock Pattern matching for $321$-avoiding permutations.
\newblock In Y.~Dong, D.-Z. Du, and O.~Ibarra, editors, {\em Algorithms and
  Computation, Proceedings of the Twentieth International Symposium (ISAAC)},
  volume 5878 of {\em Lecture Notes in Comput. Sci.}, pages 1064--1073.
  Springer, Berlin, Germany, 2009.

\bibitem{ibarra:finding-pattern:}
L.~Ibarra.
\newblock Finding pattern matchings for permutations.
\newblock {\em Inform. Process. Lett.}, 61(6):293--295, 1997.

\bibitem{jelinek:hardness-of-per:}
V.~Jel{\'{\i}}nek and J.~Kyn{\v{c}}l.
\newblock Hardness of permutation pattern matching.
\newblock arXiv:1608.00529 [math.CO].

\bibitem{neou:pattern-matchin:}
B.~E. Neou, R.~Rizzi, and S.~Vialette.
\newblock Pattern matching for separable permutations.
\newblock In S.~Inenaga, K.~Sadakane, and T.~Sakai, editors, {\em String
  Processing and Information Retrieval: 23rd International Symposium, SPIRE
  2016}, volume 9954 of {\em Lecture Notes in Comput. Sci.}, pages 260--272.
  Springer, Berlin, Germany, 2016.

\bibitem{stankova:forbidden-subse:}
Z.~E. Stankova.
\newblock Forbidden subsequences.
\newblock {\em Discrete Math.}, 132(1-3):291--316, 1994.

\bibitem{vatter:permutation-cla:}
V.~Vatter.
\newblock Permutation classes.
\newblock In M.~B{\'o}na, editor, {\em Handbook of Enumerative Combinatorics},
  pages 754--833. CRC Press, Boca Raton, Florida, 2015.

\bibitem{yugandhar:parallel-algori:}
V.~Yugandhar and S.~Saxena.
\newblock Parallel algorithms for separable permutations.
\newblock {\em Discrete Appl. Math.}, 146(3):343--364, 2005.

\end{thebibliography}
\label{sec:biblio}

\end{document}